\documentclass[12pt]{article}

\title{Cycles through two edges in signed graphs}

\author{
   Matt DeVos\thanks{
     Email: {\tt mdevos@sfu.ca}. 
     Supported by an NSERC Discovery Grant (Canada)}
 \and
   Kathryn Nurse\thanks{
  Email: {\tt knurse@sfu.ca}. Partially supported by NSERC (Canada).
}
}

\date{}

\usepackage{fullpage}
\usepackage[hidelinks]{hyperref}
\usepackage{graphicx}
\usepackage{geometry}
\usepackage{amsfonts}
\usepackage{amsmath}
\usepackage{enumerate}
\usepackage{amsthm}
\usepackage{hyperref}
\usepackage[dvipsnames]{xcolor}

\begin{document}

\maketitle

\begin{abstract}
We give a characterization of when a signed graph $G$ with a pair of distinguished edges $e_1, e_2 \in E(G)$ has the property that all cycles containing both $e_1$ and $e_2$ have the same sign.  This answers a question of Zaslavsky.
\end{abstract}

\setcounter{page}{1} \setcounter{section}{0}
\newtheorem{theorem}{Theorem}[section]
\newtheorem{lemma}[theorem]{Lemma}
\newtheorem{corollary}[theorem]{Corollary}
\newtheorem{proposition}[theorem]{Proposition}
\newtheorem{observation}[theorem]{Observation}
\newtheorem{definition}[theorem]{Definition}
\newtheorem{claim}{Claim}
\newtheorem{conjecture}[theorem]{Conjecture}
\newtheorem{problem}[theorem]{Problem}

\date{}

\section{Introduction}

Throughout we assume (signed) graphs to be finite and loopless (loops add nothing to the problem under consideration), but we permit parallel edges. 
A \emph{signed} graph is a triple $G = (V,E,\sigma)$ where $(V,E)$ is a graph and $\sigma : E \rightarrow \{-1,1\}$ is a \emph{signature}.  We say that the \emph{sign} of a cycle $C \subseteq G$ is \emph{positive} \emph{(negative)} if $\sigma(C) = \prod_{e \in E(C)} \sigma(e)$ is equal to 1 $(-1)$.  If all cycles of $G$ are positive, then we call $G$ \emph{balanced} and otherwise we call $G$ \emph{unbalanced}.

In a $2$-connected signed graph $G$, a single edge $e$ appears in cycles of both signs if and only if $G-e$ is unbalanced.  For the ``only if'' direction, let $C_1,C_2$ be cycles of opposite sign containing $e$ and note that the symmetric difference of $E(C_1)$ and $E(C_2)$ is a set of edges with negative sign and even degree at every vertex (which can thus be expressed as a disjoint union of edge sets of cycles).  For the ``if'' direction, let $e = uv$, choose a negative cycle $C$ in $G-e$, and apply Menger to choose two vertex disjoint paths from $\{u,v\}$ to $V(C)$; these two paths together with $C$ and $e$ contain the desired cycles.

Our objective in this article is to extend this simple property to a pair of edges.  If $G$ is a signed graph and $e_1, e_2 \in E(G)$, then we say that $e_1$ and $e_2$ are \emph{untied} if there exist cycles containing $e_1$ and $e_2$ of both positive and negative sign, and otherwise we say that $e_1$ and $e_2$ are \emph{tied}.  Our main result is as follows.


\begin{theorem}
\label{main}
Let $G$ be a 3-connected signed graph and let $e_1,e_2 \in E(G)$ be distinct and not in parallel with any other edges.  Then $e_1$ and $e_2$ are tied in $G$ if and only if one of the following holds:
\begin{enumerate}
\item There exists a parallel class $F$ containing edges of both signs so that $F^+ = F \cup \{e_1, e_2\}$ is an edge-cut and $G - F^+$ is balanced,
\item $e_1, e_2$ are incident with a common vertex $v$ and $G-v$ is balanced,
\item $G - \{e_1, e_2\}$ is balanced.
\end{enumerate}
\end{theorem}

In Section~2, we provide a reduction that allows us to determine the structure of arbitrary signed graphs that are tied, meaning this result implies a full characterization of when all cycles through two given edges of a signed graph have the same sign. This problem was explicitly asked by Zaslavsky in \cite[E2]{ZaslavskyThomas2017NCiS}, but let us remark that our motivation for this work is a forthcoming application of these results in the setting of nowhere-zero flows on signed graphs.

Theorem~\ref{main} may be viewed as a signed graph generalization of the following result from Lov\'asz's problem book \cite[6.67]{LovaszCombProbExercises}.  By replacing the edge $e_3$ of Theorem \ref{lovasz} with two parallel edges, one of each sign, forming a signed graph with exactly one negative edge, one observes that Theorem \ref{main} does indeed imply Theorem \ref{lovasz}. 

\begin{theorem}\label{lovasz}[Lov\'{a}sz]
    Let $G$ be a simple $3$-connected graph and $e_1,e_2,e_3 \in E(G)$ be distinct. Then there is no cycle containing $e_1,e_2,e_3$ if and only if one of the following holds:
    \begin{enumerate}
        \item $G - \{e_1,e_2,e_3\}$ is disconnected,
        \item $e_1,e_3,e_3$ are incident with a common vertex.
    \end{enumerate}
\end{theorem}

\vspace{-.1cm}

Another generalization of Theorem \ref{lovasz} is the following conjecture by Lov\'{a}sz \cite{lovasz1974} and Woodall \cite{WoodallD.R1977Ccse} (independently): 
If $G$ is a $k$-connected graph, and $S \subseteq E(G)$ a set of $k$ independent edges so that either $k$ is even or $G-S$ is connected, then there is a cycle $C \subseteq G$ with $S \subseteq E(C)$. Kawarabayashi \cite{KawarabayashiKen-ichi2002OoTD} showed that $S$ is always contained in either one cycle or two vertex-disjoint cycles. And Thomassen and H\"{a}ggkvist \cite{HaggkvistRoland1982Ctse} showed that the conjecture holds if one assumes $G$ is ($k+1$)-connected.  The following well-known conjecture of Lov\'asz also concerns connectivity, paths and cycles.


\begin{conjecture}\label{rem path}[Lov\'{a}sz]
    For any natural number $k$, there exists a least natural number $f(k)$ so that for any $f(k)$-connected graph $G$ and any $x,y \in V(G)$ there exists an induced $xy$-path $P$ so that $G-V(P)$ is $k$-connected.
\end{conjecture}

\vspace{-.1cm}

The above conjecture also has a natural generalization to signed graphs that we state below.  To deduce \ref{rem path} from \ref{signedRemPathConj}, simply add a single negative edge $xy$ to the graph (treat all other edges as positive).  

\begin{conjecture}\label{signedRemPathConj}
    For any natural number $k$, there exists a least natural number $f'(k)$ so that for any $f'(k)$-connected, unbalanced, signed graph $G$ there exists an induced negative cycle $C$ so that $G-V(C)$ is $k$-connected.
\end{conjecture}

Concerning the two conjectures above, Tutte \cite{TutteW.T.1963HtDa} proved the simplest of these cases, that $f(1)= f'(1) =3$. Using Tutte's language, a cycle $C$ in a graph $G$ is \emph{peripheral} if $C$ is induced and $G-V(C)$ is connected. Tutte showed that every $3$-connected graph has a peripheral cycle through any given edge, so $f(1) = 3$.  Moreover, he proves that the peripheral cycles generate the cycle space. That is to say that the peripheral cycles are not contained in any codimension $1$ subspace of the cycle space. It follows that every signed graph with a non-trivial signature has a negative peripheral cycle, and $f'(1) = 3$. 
Kriesell \cite{KriesellMatthias2001Ipi5} and independently Chen Gould and Yu \cite{ChenGuantao2003Gcap} show that $f(2) = 5$. 

And so we have provided two examples of interesting statements about graphs which have a natural and more general interpretation in the setting of signed graphs. 

\section{Reduction to 3-connected}

In the setting of signed graphs we are principally focused on the signs of cycles and not those of edges.  Accordingly, if $\sigma' : E \rightarrow \{-1,1\}$ we call $\sigma'$ \emph{equivalent} to $\sigma$ if every cycle $C \subseteq G$ satisfies $\sigma(C) = \sigma'(C)$.  If $S$ is an edge-cut of $G$ and we modify $\sigma$ by switching the sign of every edge in $S$, it is apparent that the resulting signature is equivalent to $\sigma$ and we call this a \emph{switch}.  On the other hand, if $\sigma$ and $\sigma'$ are equivalent, the set $S$ of all edges for which $\sigma(e) \neq \sigma'(e)$ must meet every cycle in a set of even size, and it follows from this that $S$ is an edge-cut.  So two  signatures are equivalent if and only if one can be obtained from the other by a switch. 

A $k$-\emph{separation} of a graph $G$ is a pair of subgraphs $(G_1,G_2)$ so that $E(G_1) \cap E(G_2) = \emptyset$,  $E(G_1) \cup E(G_2) = E(G)$, and $|V(G_1) \cap V(G_2)| = k$.  We say that the separation is \emph{proper} if $V(G_1) \setminus V(G_2) \neq \emptyset \neq V(G_2) \setminus V(G_1)$.  

The concept of edges $e_1$, $e_2$ being tied is vacuous if $e_1$ and $e_2$ are in separate blocks, so it suffices to consider 2-connected graphs.  The following observation will be helpful in making further reductions.

\begin{observation}
Let $G$ be a 2-connected signed graph, let $e_1, e_2 \in E(G)$, and let $(G_1, G_2)$ be a 2-separation of $G$ with $V(G_1) \cap V(G_2) = \{u,v\}$.  For $i=1,2$ let $G_i^+$ be obtained from $G_i$ by adding a positive edge $f_i$ with ends $u,v$.  
\begin{enumerate}
\item If $e_i \in E(G_i)$ for $i=1,2$ then $e_1$ and $e_2$ are tied in $G$ if and only if $e_i$ and $f_i$ are tied in $G_i^+$ for $i=1,2$.
\item If $e_1, e_2 \in E(G_1)$ and every edge in $E(G_2)$ is positive, then $e_1$ and $e_2$ are tied in $G$ if and only if they are tied in $G_1^+$.
\item If $e_1, e_2 \in E(G_1)$ and $G_2$ is unbalanced, then $e_1$ and $e_2$ are tied in $G$ if and only if they are tied in the graph obtained from $G_1^+$ by adding a negative edge $f_1'$ in parallel with $f_1$.
\end{enumerate}
\end{observation}

\begin{proof}
The first and second part follow immediately from the definitions.  The last part is a consequence of the observation that in this case the graph $G_2$ contains a negative cycle $C$ together with two vertex disjoint paths between $V(C)$ and $\{u,v\}$.  
\end{proof}

If we are interested in determining whether two edges in a given 2-connected signed graph $G$ are tied and $G$ has a proper 2-separation, the above observation allows us (possibly after resigning) to reduce the problem to one on smaller graphs.  Continuing in this manner, we may reduce the problem to the setting of 3-connected signed graphs.  Hence Theorem \ref{main} gives a complete answer.  

All of the steps in our reduction are reversible, so we can turn this around and provide a generic construction of signed graphs where two given edges are tied by taking the three types given in the above theorem and combining them as in the observation.  The possible structures of all such graphs can readily be determined but we found no better way of describing them than by way of the decompositions presented here.

\section{Proofs}

Let $G = (V,E,\sigma)$ be a signed graph and let $e \in E$ ($v \in V$).   To \emph{delete} the edge $e$ (vertex $v$) we remove this edge (vertex and all incident edges) from the graph and adjust the domain of $\sigma$ to remove these lost edges; we denote this new signed graph by $G \setminus e$ $(G \setminus v)$.  To \emph{contract} the edge $e$, first modify $\sigma$ by switching on an edge-cut (if necessary) so that $\sigma(e) = 1$, and then modify the graph by contracting $e$ and removing $e$ from the domain of $\sigma$; we denote the resulting signed graph by $G / e$.  If $H$ is a signed graph obtained from $G$ by a (possibly empty) sequence of edge and vertex deletions and edge contractions, we call $H$ a \emph{minor} of $G$.  Note that whenever $C \subseteq H$ is a cycle, there is a corresponding cycle $C^* \subseteq G$ containing all edges in $C$ and having the same sign as $C$.  In particular, this implies the following key property.

\begin{observation}
Let $H$ be a minor of the signed graph $G$.  If $e_1,e_2$ are untied edges of $H$, then they are also untied in $G$.  
\end{observation}

Next we will introduce three signed graphs (actually families of signed graphs) each of which has a distinguished cycle $C$ that is negative together with distinguished edges $e_1, e_2$.  We call a signed graph a \emph{hat} if it consists of a length two negative cycle $C$ with vertices $x_1,x_2$ plus one additional vertex $y$ and the edges $e_i = x_i y$ for $i=1,2$.  We call a signed graph an \emph{target} if it consists of a length four negative cycle $C$ with cyclically ordered vertices $x_1, \ldots, x_4$ together with the edges $e_1 = x_1x_3$ and $e_2 = x_2 x_4$.  Finally, we call a signed graph a \emph{hedgehog} if it consists of a negative cycle $C$ with vertices $x_1, x_2, x_3$ together with two additional vertices $y_1, y_2$ and all edges between $\{y_1,y_2\}$ and $\{x_1, x_2, x_3\}$, and here we define $e_i = x_i y_i$ for $i=1,2$.  

\begin{center}
    \includegraphics[height=3.5cm]{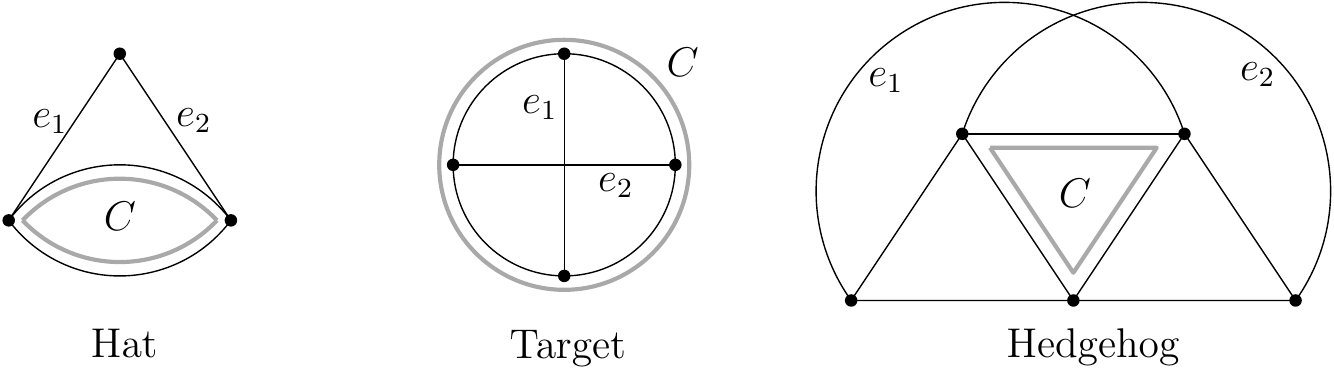}
\end{center}


For all three of the types of graphs above, the  number of cycles containing both $e_1$ and $e_2$ is even, and the symmetric difference of the edge sets of these cycles is $E(C)$ which is negative.  It follows from this that in these graphs $e_1$ and $e_2$ are untied.  The heart of our argument is to show that if our graph is not one of the named counterexamples to Theorem \ref{main}, then it contains a hat, target, or hedgehog graph as a minor.

Our arguments lean on working with a carefully chosen negative cycle $C$ in the graph and we require some notation for the manner in which the rest of the graph attaches to $C$.  For this purpose we adopt Tutte's notation that we introduce next. Let $G$ be a graph and let $H \subseteq G$.  A \emph{bridge} of $H$ is a subgraph of $G \setminus E(H)$ of one of the two forms:
\begin{itemize}
\item A single edge $uv$ (and its ends) where $u,v \in V(H)$ and $uv \not\in E(H)$
\item A component $F$ of $G - V(H)$ together with all edges of $G$ with exactly one end in $V(F)$
\end{itemize}
If $B$ is a bridge of $H$, we call a vertex in $V(B) \cap V(H)$ an \emph{attachment} vertex of $B$.



\begin{lemma}\label{mainLem}
Let $G = (V,E,\sigma)$ be a simple signed 3-connected graph, and let $e_1,e_2 \in E(G)$ be nonadjacent.  If there exists a negative cycle in $G - \{e_1,e_2\}$, then $e_1$ and $e_2$ are untied.
\end{lemma}

\begin{proof}

Suppose for contradiction the lemma is false, and let $G$ be a counterexample so that $|V|$ is minimum.  Choose a negative cycle $C \subseteq G \setminus \{e_1,e_2\}$ subject to the following constraints: Both $e_1$ and $e_2$ are in the same bridge of $C$ if possible, subject to this the bridge of $C$ containing $e_1$ is maximum, subject to this the bridge of $C$ containing $e_2$ is maximum, and subject to this the lexicographic ordering of the sizes of the other bridges is maximized.  We proceed by proving a sequence of claims to establish the result.  

\bigskip

\noindent{(1)} Every bridge of $C$ must contain $e_1$ or $e_2$.  

\bigskip

Suppose (for a contradiction) that this does not hold and consider the smallest bridge $B$ not containing $e_1$ or $e_2$.  If $C$ has length 3, then the graph $G' = G \setminus (V(B) \setminus V(C))$ is still 3-connected and is a smaller counterexample (to see that $G'$ is 3-connected, note that any proper $k$-separation in $G'$ for $k < 3$ must have all vertices in $V(C)$ in one of the subgraphs and thus extends to a $k$-separation in $G$).  It follows that $|V(C)| \ge 4$.  We claim that there must exist distinct vertices $x_1,x_2,x_3,x_4 \in V(C)$ appearing in this cyclic order on $C$ so that $x_1,x_3 \in V(B)$ and both $x_2$ and $x_4$ are contained in bridges other than $B$. If $V(C) \subseteq V(B)$, then the 3-connectivity (and simplicity) of $G$ imply that the bridge containing $e_1$ has two non-consecutive attachments and this gives the desired vertices.  Otherwise there exists $x \in V(C)$ that is not an attachment of $B$. Choose a minimal path $P \subseteq C-x$ with ends $y,y'$ so that all attachments of $B$ lie in $V(P)$.  Then $P$ has length at least $3$, and $\{y,y'\} \in V(B)$. If no bridge other than $B$ has an attachment in $V(P) \setminus \{y,y'\}$, then this contradicts the 3-connectivity of $G$ (as deleting $y,y'$ would disconnect the graph). Therefore the vertices $x_1, \ldots, x_4$ may be chosen (as claimed).  Next, choose a path $Q \subseteq B$ with ends $x_1,x_3$ and let $R,R'$ be the two paths in $C$ with ends $x_1,x_3$.  One of $R \cup Q$ or $R' \cup Q$ must be a negative cycle for which the bridge containing $x_2$ or $x_4$ has increased in size (and no bridge apart from $B$ has decreased).  This contradicts our choice of $C$, thus proving (1).

\bigskip

\noindent{(2)} No bridge contains $e_1$ and $e_2$.  

\bigskip

Suppose (for a contradiction) that (2) is false and let $B$ be a bridge of $C$ with $e_1, e_2 \in E(B)$.  Note that by (1), $B$ must be the unique bridge of $C$.  Call a path $P \subseteq B$ \emph{good} if $e_1,e_2 \in E(P)$, both ends of $P$ are in $V(B) \cap V(C)$, but no interior vertex of $P$ lies in $V(B) \cap V(C)$.  If there is a good path $P$, then $C \cup P$ has a hat-minor and $e_1,e_2$ are untied, giving a contradiction.  So it will suffice to prove that a good path exists.  Let $B' = B \setminus V(C)$. If there exists a cycle $C' \subseteq B'$ with $e_1,e_2 \in E(C')$ then we may apply Menger's theorem to choose three vertex disjoint paths between $C$ and $C'$ and these paths together with $C'$ contain a good path.  It follows from this (and the non-adjacency of $e_1, e_2$) that $e_1$ and $e_2$ are in different blocks of the graph $B'' = B' \cup \{e_1, e_2\}$.  Choose a maximal list of blocks $H_1, \ldots, H_k$ in the graph $B''$ so that $V(H_i) \cap V(H_{i+1}) \neq \emptyset$ for $1 \le i \le k-1$ and so that $e_1$ and $e_2$ are both contained in some block from this list.  If $V(H_1) \cap V(C) \neq \emptyset$ $(V(H_k) \cap V(C) \neq \emptyset)$ then we denote this set by $X$ $(Y)$.  Note that in this case the set must consist of a single vertex, and the block consists of just one edge $e_1$ or $e_2$ (this happens only when $e_i \not\in E(B'))$.  Otherwise, it follows from the 3-connectivity of $G$ that the set $X$ $(Y)$ of vertices in $V(C)$ adjacent to a vertex in $V(H_1) \setminus V(H_2)$ $(V(H_k) \setminus V(H_{k-1}))$ has size at least 2.  It now follows from our assumptions that we may choose distinct vertices $x \in X$ and $y \in Y$.  Now using the structure of our list of blocks it is straightforward to choose a good path from $x$ to $y$, and this completes the proof of (2).

\bigskip

In the remainder of our proof we know from (1) and (2) that for $i=1,2$ there is a bridge $B_i$ containing $e_i$ and that $C$ has exactly two bridges, $B_1$ and $B_2$.  Mimicking our earlier notation, we will call a path $P \subseteq B_i$ \emph{good} if it has both ends in $V(C)$, is internally disjoint from $V(C)$, and contains the edge $e_i$.  
We need just one added bit of notation before the next claim.  
Let $\Gamma$ be a graph and let $x, x_1, x_2, x_3 \in V(\Gamma)$ be distinct.  Let $P_1, P_2, P_3$ be internally disjoint paths so that $P_i$ has ends $x$ and $x_i$ for $1 \le i \le 3$.  In this case we call the graph $T = \bigcup_{i=1}^3 P_i$ a $\{x_1, x_2, x_3\}$-\emph{tripod} and we call $P_i$ the $x_i$-\emph{leg} of~$T$.  

\bigskip

\noindent{(3)} $e_1$ is not incident with a vertex of $C$.

\bigskip

Suppose (for a contradiction) that (3) is false and let $u \in V(C)$ be incident with $e_1$.  Define $P$ to be the minimal subpath of $C \setminus u$ with the property that $V(P)$ contains all attachment vertices of $B_2$ except for (possibly) $u$.  First suppose there exists an attachment vertex of $B_1$, say $v$ in the interior of $P$.  In this case we may choose a good path $Q_1 \subseteq B_1$ from $u$ to $v$.  Next choose a path $Q_2 \subseteq B_2$ internally disjoint from $V(C)$ with the same ends as $P$.  If $Q_2$ is a good path, then $C \cup Q_1 \cup Q_2$ contains a target-minor and $e_1,e_2$ are untied.  Otherwise there is a negative cycle contained in $C \cup Q_2$ that contradicts the choice of $C$ (as it improves the size of $B_1$).  So the interior of $P$ does not contain an attachment of $B_1$, and by $3$-connectivity of $G$ it must be that $u$ is an attachment vertex of $B_2$.  Let $P'$ be the minimal subpath of $C - E(P)$ that contains all attachment vertices of $B_1$ and let $w,z$ be the ends of $P'$.  If one of $w,z$ is equal to $u$, then $\{w,z\}$ would be a 2-vertex cut contradicting our connectivity assumption, so we must have $w,z \neq u$.  It follows from the connectivity of $B_1' = B_1 \setminus V(C)$ that we may choose a $\{u,w,z\}$-tripod $T_1 \subseteq B_1$ with $V(T_1) \cap V(C) = \{u,w,z\}$ so that $e_1$ is contained in the $u$-leg of $T_1$.  (To see this, let $e_1 = uu'$ and choose vertices $w', z' \in V(B_1')$ respectively adjacent to $w, z$ and note that by the connectivity of $B_1'$ this graph must either contain a path containing $u',w',z'$ or it must contain a $\{u',w',z'\}$-tripod.)  Since $e_2$ is not incident with $u$, the 3-connectivity of $G$ implies the existence of a good path $R \subseteq B_2$ with ends $u', u''$ and $u \neq u', u''$.  It now follows from the connectivity of $B_2 \setminus V(C)$ that $R$ can be enlarged to a $\{u,u', u''\}$-tripod $T_2$ with $V(T_2) \cap V(C) = \{u,u',u''\}$ and with the property that $e \in E(T_2)$ but $e$ is not in the $u$-leg of $T_2$.  Now $C \cup T_1 \cup T_2$ contains a hedgehog minor, so $e_1, e_2$ are untied.  

\bigskip

\noindent{(4)} $|V(C)| \ge 4$.

\bigskip

Suppose for a contradiction that (4) is false and let $V(C) = \{u,u',u''\}$.  If $e_2$ is incident with a vertex in $V(C)$, then assume (without loss) that this vertex is $u$.  By the 3-connectivity of $G$ we may choose a path $P \subseteq B_1 \setminus u$ from $u'$ to $u''$ that contains $e_1$.  This path may be extended to a $\{u,u',u''\}$ tripod $T \subseteq B_1$ with the property that $e \in E(T)$ but $e$ is not in the $u$-leg of $T$.  Now form a signed graph minor $G'$ of $G$ by deleting all edges in $E(B_1) \setminus E(T)$, then deleting all isolated vertices, and then contracting edges in $T$ to reduce each leg to a single edge (in the leg containing $e_1$ contract all edges other than $e_1$).  The resulting signed graph $G'$ is 3-connected and therefore contradicts the choice of $G$.  

\bigskip

With this last claim in place, we proceed to the end of the proof.  For this task we will devise a system of labels assigned to vertices in $V(C)$ that will permit us to route good paths in $B_1$.  Let $H$ be the block of $B_1 \setminus V(C)$ that contains the edge $e_1$.  If $v \in V(C) \cap V(B_1)$ and $x \in V(H)$ we assign $v$ the label $x$ if there exists a path in the graph $B_1 \setminus E(H)$ from $v$ to $x$ that is internally disjoint from $V(C)$.  If $H$ consists of the single edge $e_1 = xy$, then (by 3-connectivity) there are at least two vertices with label $x$ and at least two with label $y$.  On the other hand if $H$ is a larger block, then (by 3-connectivity) there are at least 3 distinct labels appearing on vertices in $V(C)$.  Note that whenever $u,v \in V(C)$ are distinct vertices with distinct labels, there exists a good path from $u$ to $v$ in $B_1$.  

We claim that there exist four distinct vertices $v_1, v_2, v_3, v_4$ appearing in this cyclic order around $C$ so that $v_1,v_3$ are attachment vertices of $B_1$ and $v_2,v_4$ are attachment vertices of $B_2$.  If every vertex in $V(C)$ is an attachment of $B_1$, then this is straightforward to verify.  Otherwise choose $v_2 \in V(C) \setminus V(B_1)$ and let $P$ be the minimal path in $C - v_2$ containing all attachments of $B_1$.  If there does not exist an internal vertex of $P$ that is an attachment of $B_2$, then deleting the two ends of $P$ disconnects the graph --- a contradiction.  Therefore there is a vertex $v_4$ in the interior of $P$ that is an attachment of $B_2$ and taking $v_1,v_3$ to be the ends of $P$ give the desired vertices.

Let $Q,Q'$ be the two paths in $C$ with ends $v_2,v_4$.  First suppose that there exist $v \in V(Q) \setminus \{v_2,v_4\}$ and $v' \in V(Q') \setminus \{v_2,v_4\}$ so that $v$ and $v'$ both have distinct labels.  In this case we may choose a good path $R_1 \subseteq B_1$ from $v$ to $v'$.  Let $R_2 \subseteq B_2$ be a path internally disjoint from $V(C)$ from $v_2$ to $v_4$.  If $R_2$ is good, then $C \cup R_1 \cup R_2$ contains a target-minor and we are done.  Otherwise, $C \cup R_2$ contains a negative cycle that contradicts the choice of $C$ (as it improves the size of $B_1$).  Therefore, we may assume that there exists a vertex $x \in V(H)$ so that the only label appearing on interior vertices of $Q$ and $Q'$ is $x$.  It follows from this that $v_2$ has label $y$ and $v_4$ has label $y'$ for some $y,y' \neq x$.  By the definition of the labels, for $i=1,3$ we may choose a path $P_i \subseteq B_1$ starting at the vertex $v_i$ and ending at the vertex $x$ so that $V(P_i) \cap V(C) = \{v_i\}$ and so that $V(P_i) \cap V(H) = \{x\}$.  Now the graph $C \cup P_1 \cup P_3$ contains a negative cycle $C'$ with the property that $e_1$ and $e_2$ are in the same bridge of $C'$, thus contradicting the choice of $C$ (to see this, note that we may choose such a negative cycle $C'$ with $v_j \not\in V(C')$ for $j=2$ or $j=4$, and then both $e_1$ and $e_2$ will be in the bridge of $C'$ containing $v_j$).  This final contradiction completes the proof.
\end{proof}

With this lemma in hand we are ready to prove the main result.

\begin{proof}[Proof of Theorem \ref{main}]
The ``if'' direction is straightforward to verify.  For the ``only if'' direction, first suppose that $e_1$ and $e_2$ are incident with a common vertex $u$, say $e_i = u v_i$ for $i=1,2$.  If $G- u$ is balanced, then we have the second structure from the theorem statement.  Otherwise we may choose an unbalanced cycle $C \subseteq G - u$ and (by Menger) two vertex disjoint paths from $V(C)$ to $\{v_1, v_2\}$.  The existence of these subgraphs implies a hat-minor.  So, we may now assume that $e_1$ and $e_2$ are nonadjacent.  

Next suppose that there exist two parallel edges $f,f'$ of opposite sign. If $f,f'$ are incident with an end of $e_1$ or $e_2$, say $e_1 = uv$, $e_2 = xy$, and $f = uw$. Then (by Menger) there are two disjoint paths from $\{v,w\}$ to $\{x,y\}$ in $G \backslash u$ which implies the existence of a hat-minor. Otherwise, by Theorem \ref{lovasz} (or by an earlier result of Watkins and Mesner \cite{watkins1967cycles}) either there is a cycle containing $e_1,e_2,f$ which means a hat-minor, or $G-S$ is disconnected, where $S$ is the set of edges containing $e_1,e_2$ and all edges parallel to $f$. If $G - S$ does not contain a negative cycle, we have the first structure of the theorem statement.  If it does, choose such a cycle $C$ and choose a vertex $z$ in the component of $G - S$ not containing $C$.  It follows from the 3-connectivity of $G$ that we may choose two paths $P_1, P_2 \subseteq G \backslash u$ starting at $z$ and ending in $V(C)$ with $V(P_1) \cap V(P_2) = \{z\}$.  Now $C \cup P_1 \cup P_2$ contains a hat-minor.

So we may now assume that $G$ does not contain a negative cycle of length~2.  If $G$ has any parallel edges, then they all have the same sign and we may delete all but one of these edges without effect.  Thus we may now assume that $G$ is simple and that $e_1, e_2$ are non-adjacent.  If $G - \{e_1, e_2\}$ is balanced, then we have the third structure from the theorem statement.  Otherwise it follows from Lemma~\ref{mainLem} that $e_1$ and $e_2$ are not tied in $G$, and this completes the proof.
\end{proof}

\bibliographystyle{abbrv}
\bibliography{bib}

\end{document}